\documentclass{amsart}
\usepackage{amsmath,amssymb,enumerate,amstext,
amscd}
\usepackage[all]{xy}
\newtheorem{lem}{Lemma}[section]
\newtheorem{prop}[lem]{Proposition}
\newtheorem{cor}[lem]{Corollary}
\newtheorem{thm}[lem]{Theorem}
\newtheorem{defin}[lem]{Definition}
\newtheorem{remark}[lem]{Remark}

\title[Hochschild cohomology]{\bf Vanishing of the 
Hochschild cohomology for some self-injective 
special biserial algebra of rank four}
\author[T.~Furuya and T.~Hayami]{Takahiko Furuya 
and Takao Hayami}
\address{Takahiko Furuya, School of Dentistry, 
Meikai University, 1-1 Keyakidai, Sakado, Saitama, 
Japan}
\email{furuya@dent.meikai.ac.jp}
\address{Takao Hayami, Faculty of Engineering, 
Hokkai-Gakuen University, 4-1-40 Asahi-machi, 
Toyohira-ku, Sapporo, Japan}
\email{hayami@ma.kagu.tus.ac.jp}
\thanks{The second author is supported by a grant 
from Hokkai-Gakuen University.}
\begin{document}

\maketitle

\begin{abstract}
In this paper, we determine the dimensions of the 
Hochschild cohomology groups of some self-injective 
special biserial algebra whose Grothendieck group 
is of rank $4$. This result provides us with a 
negative answer to Happel's question in \cite{H}. 
\end{abstract}

\section{Introduction}\label{introduction}
Let $\varGamma$ be the following quiver with four 
vertices $e_{0}$, $e_{1}$, $e_{2}$, $e_{3}$, and 
eight arrows $a_{l,m}$ for $l=0,1$ and $m=0,1,2,3$: 
$$
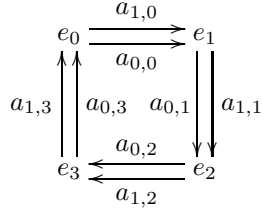
\begin{xy}
(0,0)		*+{e_3}="e3",
(18,0)		*+{e_2}="e2",
(18,18)		*+{e_1}="e1",
(0,18)		*+{e_0}="e0",
(9,21.25)	*{a_{1,0}},
(9,14.5)	*{a_{0,0}},
(23,8.5)	*{a_{1,1}},
(13.5,8.5)	*{a_{0,1}},
(9,-3.5)	*{a_{1,2}},
(9,3)		*{a_{0,2}},
(-5,8.5)	*{a_{1,3}},
(5,8.5)		*{a_{0,3}},
\ar @<1mm> "e0"; "e1",
\ar @<-1mm> "e0"; "e1",
\ar @<1mm> "e1"; "e2",
\ar @<-1mm> "e1"; "e2",
\ar @<1mm> "e2"; "e3",
\ar @<-1mm> "e2"; "e3",
\ar @<1mm> "e3"; "e0",
\ar @<-1mm> "e3"; "e0",
\end{xy}$$
We consider the subscript $i$ of $e_{i}$ as modulo 
$4$, and the subscripts $l$ and $m$ of $a_{l,m}$ 
as modulo $2$ and $4$, respectively. Therefore, 
for all $l,m\in \mathbb{Z}$, the arrow $a_{l,m}$ 
starts at the vertex $e_{m}$ and ends with the 
vertex $e_{m+1}$. We write paths from left to right. 

Let $K$ be an algebraically closed field, and denote 
by $K\varGamma$ the path algebra of $\varGamma$ 
over $K$. We set $x_{l}:=\sum_{m=0}^{3}a_{l,m}\in 
K\varGamma$ for $l=0,1$. (Hence we may consider 
the subscript $l$ of $x_{l}$ as modulo $2$.) Then, 
for any $i,l\in\mathbb{Z}$ and an integer $k\geq0$, 
the element $e_{i}x_{l}^{k}$ is exactly the path 
$a_{l,i}a_{l,i+1}\cdots a_{l,i+k-1}$ of length $k$, 
and hence $e_{i}x_{l}^{k}=e_{i}x_{l}^{k}e_{i+k}=
x_{l}^{k}e_{i+k}$ holds. 

Let $T\geq0$ be an integer, and let $q_{i}\in 
K^{\times}$ $(=K\backslash\{0\})$ for $0\leq i\leq3$. 
We always consider the subscript $i$ of $q_{i}$ as 
modulo $4$. Let $I=I_{T}(q_{0},q_{1},q_{2},q_{3})$ 
denote the ideal in $K\varGamma$ generated by the 
uniform elements $e_{i}x_{0}x_{1}$, $e_{i}x_{1}x_{0}$, 
$e_{j}(q_{j}x_{0}^{4T+2}+x_{1}^{4T+2})$, and $e_{k}
(q_{k}x_{1}^{4T+2}+x_{0}^{4T+2})$ for $0\leq i\leq3$, 
$j=0,2$ and $k=1,3$, namely, $I=I_{T}(q_{0},q_{1},
q_{2},q_{3}):=\langle x_{l}x_{l+1},\ e_{i}(q_{i}
x_{i}^{4T+2}+x_{i+1}^{4T+2})\ \big|\ l=0,1;\ 0\leq 
i\leq3\rangle$. We define the algebra $A=A_{T}(q_{0},
q_{1},q_{2},q_{3})$ to be the quotient algebra $K
\varGamma/I_{T}(q_{0},q_{1},q_{2},q_{3})$. Then $A=
A_{T}(q_{0},q_{1},q_{2},q_{3})$ is a self-injective 
special biserial algebra, and hence is of tame 
representation type. In particular, if $T=0$, then 
$A=A_{0}(q_{0},q_{1},q_{2},q_{3})$ is a Koszul 
self-injective algebra for all $q_{i}\in K^{\times}$ 
$(0\leq i\leq3)$ (see Proposition~\ref{koszul}). 

In \cite{F2}, we have found an explicit $K$-basis for 
the Hochschild cohomology group ${\rm HH}^{i}(A)$ 
($i\geq0$) of $A=A_{T}(1_{K},1_{K},1_{K},1_{K})$ 
for all $T\geq0$, and gave a presentation by 
generators and relations of the Hochschild cohomology 
ring modulo nilpotence ${\rm HH}^{*}(A)/\mathcal{N}_{A}$ 
of $A=A_{0}(1_{K},1_{K},1_{K},1_{K})$, where 
$\mathcal{N}_{A}$ denotes the ideal in ${\rm HH}^{*}(A)$ 
generated by all homogeneous nilpotent elements. 
This result shows that, for $T=0$ and $q_{i}=1_{K}$ 
$(0\leq i\leq3)$, ${\rm HH}^{*}(A)/\mathcal{N}_{A}$ 
is finitely generated as an algebra. In this paper, 
we construct an explicit projective bimodule 
resolution of $A=A_{T}(q_{0},q_{1},q_{2},q_{3})$ 
by using a certain subset $\mathcal{G}^{n}$ 
($n\geq0$) of $K\varGamma$ found in \cite{GSZ} 
(see Theorem~\ref{resolution_of_A}), and then 
compute the dimension of ${\rm HH}^{i}(A)$ ($i\geq0$) 
completely for $T\geq0$ and $q_{i}$ ($0\leq i\leq3$) 
such that the product $q_{0}q_{1}q_{2}q_{3}\in 
K^{\times}$ is not a root of unity (see 
Theorem~\ref{main_thm}~(a)). 

In \cite{H}, Happel has asked whether, if the 
Hochschild cohomology groups ${\rm HH}^{n}(\varLambda)$ 
of a finite-dimensional $K$-algebra $\varLambda$ 
vanish for all $n\gg0$, then the global dimension 
of $\varLambda$ is finite. Recently, in the papers 
\cite{BGMS, BE, PS, XZ}, a negative answer to this 
question have been obtained, where the authors 
studied the Hochschild cohomology groups for certain 
self-injective special biserial Koszul algebras. 
In this paper, we verify that some of our algebras 
$A$ also give a negative answer to this question. 
In fact, the main theorem says that, if the product 
$q_{0}q_{1}q_{2}q_{3}\in K^{\times}$ is not a root of 
unity, then ${\rm HH}^{n}(A)$ vanish for all $n\geq3$ 
if and only if $T=0$ (see Theorem~\ref{main_thm}~(b)). 

Throughout this paper, for any arrow $\alpha$ in 
$\varGamma$, we denote its origin by $\mathfrak{o}
(\alpha)$ and its terminus by $\mathfrak{t}(\alpha)$. 
Moreover we write $\otimes_{K}$ as $\otimes$, and 
denote the enveloping algebra $A^{\rm op}\otimes A$ 
of $A$ by $A^{\rm e}$ and the Jacobson radical of 
$A$ by $\mathfrak{r}_{A}$. 

\section{A projective bimodule resolution of $A$}
\label{section_resolution}

Let $B:=K\varDelta/I$ be a finite-dimensional 
$K$-algebra with $\varDelta$ a finite quiver, and 
$I$ an admissible ideal in $K\varDelta$. Denote by 
$\mathfrak{r}_{B}$ the Jacobson radical of $B$. We 
start by recalling the construction of a set 
$\mathcal{G}^n$ ($n\geq0$) introduced in \cite{GSZ} 
for the right $B$-module $B/\mathfrak{r}_{B}$, from 
which we immediately obtain a minimal projective 
resolution of $B/\mathfrak{r}_{B}$. Let 
$\mathcal{G}^{0}$ be the set of all vertices of 
$\varDelta$, $\mathcal{G}^{1}$ the set of all arrows 
of $\varDelta$ and $\mathcal{G}^{2}$ a minimal set 
of uniform generators of $I$. Then, in \cite{GSZ}, 
Green, Solberg and Zacharia proved that, for 
$n\geq3$, we can construct a set $\mathcal{G}^{n}$ 
of uniform elements in $K\varDelta$ such that there 
is a minimal projective resolution $(P^{\bullet},d)$ 
of the right $B$-module $B/\mathfrak{r}_{B}$ satisfying 
the following conditions: 

\begin{enumerate}[(i)]
\item For all $n\geq0$, we have $P^{n}=\bigoplus_{x
\in\mathcal{G}^{n}}\mathfrak{t}(x)B$.

\item For each $x\in\mathcal{G}^{n}$, there are 
unique elements $r_{y},s_{z}\in K\varDelta$, where 
$y\in\mathcal{G}^{n-1}$ and $z\in\mathcal{G}^{n-2}$, 
such that $x=\sum_{y \in \mathcal{G}^{n-1}}yr_{y}=
\sum_{z\in\mathcal{G}^{n-2}}zs_{z}$. 

\item For all $n\geq1$, $d^{n}:P^{n}\rightarrow 
P^{n-1}$ is determined by $d^{n}(\mathfrak{t}(x)):=
\sum_{y\in\mathcal{G}^{n-1}}r_{y}\mathfrak{t}(x)$ 
for $x\in\mathcal{G}^{n}$, where $r_{y}$ denotes 
the element of (ii).
\end{enumerate}
In this section, we provide a set $\mathcal{G}^{n}$ 
$(n\geq0)$ for the right $A$-module 
$A/\mathfrak{r}_{A}$, and then use it to give a 
projective bimodule resolution $(R^{\bullet},
\partial)$ of $A=A_{T}(q_{0},q_{1},q_{2},q_{3})$. 
For $v\in\mathbb{Z}$ and $u\geq1$, we denote the 
product $\prod_{t=0}^{u-1}q_{t+v}$ in $K$ by 
$S_{u,v}$. (Thus the right lower subscript $v$ of 
$S_{u,v}$ can be considered as modulo $4$.) 
Furthermore, for our convenience, we set 
$S_{0,t}:=1_{K}$ for all $t\in\mathbb{Z}$. Note 
that $S_{r,u}S_{t,r+u}=S_{r+t,u}$, $S_{t,u}S_{t,u+1}
S_{t,u+2}S_{t,u+3}=S_{2t,v}S_{2t,v+2}=S_{4t,w}
=(q_{0}q_{1}q_{2}q_{3})^{t}$, and $S_{2t,2t+u}=
S_{2t,u+2}$ hold for all $u,v,w\in\mathbb{Z}$, 
$r\geq0$, and $t\geq0$. 
\subsection{Sets $\mathcal{G}^{n}$ for $A/
\mathfrak{r}_{A}$}

First, to construct a set $\mathcal{G}^{n}$ for each 
$n\geq0$, we define the elements $g_{i,j}^{n}$ ($0
\leq i\leq3$; $0\leq j\leq n$) in $K\varGamma$ as 
follows. Here, recall that the subscript $l$ of 
$x_{l}$ is considered as modulo $2$. 

\begin{defin}
For $0\leq i\leq3$, and $0\leq j\leq n$, we 
recursively define the uniform element $g_{i,j}^{n}$ 
in $K\varGamma$ by the following formulas{\rm:} 
\begin{enumerate}[\rm(a)]
\item If $n=0$, then $g_{i,0}^{0}:=e_{i}$. 
\item If $n=2m+1$ for $m\geq0$, then  
$$g_{i,j}^{2m+1}:=\begin{cases}
g_{i,0}^{2m}x_{i} 
		& \mbox{if $j=0$}\\ 
S_{2m-j+1,i+j-1}g_{i,j-1}^{2m}x_{i+1}^{4T+1}
+g_{i,j}^{2m}x_{i}
		& \mbox{if $1\leq j\leq m$}\\
S_{2m-j+1,i+j-1}g_{i,j-1}^{2m}x_{i+1}+g_{i,j}^{2m}
x_{i}^{4T+1}
		& \mbox{if $m+1\leq j\leq2m$}\\
g_{i,2m}^{2m}x_{i+1} 
		& \mbox{if $j=2m+1$.}
\end{cases}$$

\item If $n=2m$ for $m\geq1$, then  
$$g_{i,j}^{2m}:=\begin{cases}
g_{i,0}^{2m-1}x_{i+1} & \mbox{if $j=0$}\\
S_{2m-j,i+j-1}g_{i,j-1}^{2m-1}x_{i}^{4T+1}+
g_{i,j}^{2m-1}
x_{i+1} 	& \mbox{if $1\leq j\leq m-1$}\\
S_{m,i+m-1}g_{i,m-1}^{2m-1}x_{i}^{4T+1}+
g_{i,m}^{2m-1}x_{i+1}^{4T+1} 	& \mbox{if $j=m$}\\
S_{2m-j,i+j-1}g_{i,j-1}^{2m-1}x_{i}+g_{i,j}^{2m-1}
x_{i+1}^{4T+1}	& \mbox{if $m+1\leq j\leq 2m-1$}\\
g_{i,2m-1}^{2m-1}x_{i} 
		& \mbox{if $j=2m$.}
\end{cases}$$

\end{enumerate} 
\end{defin}
\noindent 
Note that $\mathfrak{o}(g_{i,j}^{n})=e_{i}$ for all 
$n\geq0$, $0\leq i\leq3$ and $0\leq j\leq n$, and 
also, if $i+n\equiv t$ (${\rm mod}\ 4$), then 
$\mathfrak{t}(g_{i,j}^{n})=e_{t}$. (So, it can be 
considered the left lower subscript $i$ of 
$g_{i,j}^{n}$ as modulo $4$.)

Now set 
$$\mathcal{G}^{n}:=\left\{g^{n}_{i,j}\ \big|\ 0\leq 
i\leq3;\ 0\leq j\leq n\right\}\quad\mbox{for all 
$n\geq0$.}$$
Then, noting $e_{i+2m-1}(q_{i+2m-1}x_{i+1}^{4T+2}+
x_{i}^{4T+2})=e_{i+2m-2}(q_{i+2m-2}x_{i}^{4T+2}+
x_{i+1}^{4T+2})=0$ in $A$ for $0\leq i\leq3$ and 
$m\geq1$, we see that the sets $\mathcal{G}^{n}$ 
satisfy the conditions (i), (ii), and (iii) in the 
beginning of this section. 
\begin{remark} 
%%%%%%%%%%%%%%
\rm %%%%%%%%%%
%%%%%%%%%%%%%%
It follows that $\mathcal{G}^{0}=\{e_{i}\mid0\leq i
\leq3\}$, $\mathcal{G}^{1}=\{a_{l,m}\mid l=0,1;\ 0
\leq m\leq3\}$, and $\mathcal{G}^{2}=\left\{e_{i}
x_{l}x_{l+1},\ e_{i}(x_{i+1}^{4T+2}+q_{i}x_{i}^{4T
+2})\mid l=0,1;\ 0\leq i\leq3\right\}$, so that 
$\mathcal{G}^{2}$ is precisely a minimal set of 
uniform generators of $I=I_{T}(q_{0},q_{1},q_{2},
q_{3})$. 
\end{remark}
\noindent
Now we notice, in the case where $T=0$, that the 
resolution $(P^{\bullet},d)$ determined by (i), (ii) 
and (iii) in the beginning of this section is a 
linear resolution of $A/\mathfrak{r}_{A}$, and hence 
we have the following: 

\begin{prop}\label{koszul}
For any $q_{i}\in K^{\times}$ $(0\leq i\leq3)$, 
the algebra $A=A_{0}(q_{0},q_{1},q_{2},q_{3})$ is 
a self-injective Koszul algebra. 
\end{prop}

\subsection{A projective bimodule resolution of 
$A$}
We now construct a minimal projective bimodule 
resolution $(R^{\bullet},\partial)$ for $A$ by using 
the sets $\mathcal{G}^n$. For simplicity, we denote 
by $\mathfrak{p}_{i,j}^{n}$ the element $\mathfrak{o}
(g_{i,j}^{n})\otimes\mathfrak{t}(g_{i,j}^{n})$ in 
$A\mathfrak{o}(g_{i,j}^{n})\otimes\mathfrak{t}
(g_{i,j}^{n})A$ for all $n\geq0$, $0\leq i\leq 3$, 
and $0\leq j\leq n$. (Hence we may consider the 
subscript $i$ of $\mathfrak{p}_{i,j}^{n}$ as modulo 
$4$.) 

First, for $n\geq0$, define the projective  
$A$-$A$-bimodule, equivalently right 
$A^{\rm e}$-module, $R^{n}$ by 
$$R^{n}:=\bigoplus_{g\in\mathcal{G}^{n}}A\mathfrak{o}
(g)\otimes\mathfrak{t}(g)A=\bigoplus_{i=0}^{3}
\bigg(\bigoplus_{j=0}^{n}A\mathfrak{p}_{i,j}^{n}A
\bigg).$$
Second, let $\partial^{0}:R^{0}\rightarrow A$ be 
the multiplication map, and for $n\geq1$ let 
$\partial^{n}: R^{n}\rightarrow R^{n-1}$ be the 
$A$-$A$-bimodule homomorphism determined by the 
following: 

If $n=2m+1$ for $m\geq0$, then, for $0\leq 
i\leq3$ and $0\leq j\leq 2m+1$, 
\begin{align}\label{image_2m+1}
	&\partial^{2m+1}(\mathfrak{p}_{i,j}^{2m+1})
:=\nonumber\\
&\begin{cases}\mathfrak{p}_{i,0}^{2m}x_{i}-x_{i}
\mathfrak{p}_{i+1,0}^{2m} 
	& \mbox{if $j=0$}\\[2mm]
S_{2m-j+1,i+j-1}\mathfrak{p}_{i,j-1}^{2m}
x_{i+1}^{4T+1}+\mathfrak{p}_{i,j}^{2m}x_{i}
	& \\
\quad\quad-S_{j,i}x_{i}\mathfrak{p}_{i+1,j}^{2m}-
x_{i+1}^{4T+1}\mathfrak{p}_{i+1,j-1}^{2m} 
	& \mbox{if $1\leq j<m+1$}\\[2mm]
S_{2m-j+1,i+j-1}\mathfrak{p}_{i,j-1}^{2m}x_{i+1}+
\mathfrak{p}_{i,j}^{2m}x_{i}^{4T+1}
	& \\
\quad\quad
-S_{j,i}x_{i}^{4T+1}\mathfrak{p}_{i+1,j}^{2m}-
x_{i+1}\mathfrak{p}_{i+1,j-1}^{2m} 
	& \mbox{if $m+1\leq j<2m+1$}\\[2mm]
\mathfrak{p}_{i,2m}^{2m}x_{i+1}-x_{i+1}
\mathfrak{p}_{i+1,2m}^{2m}
	& \mbox{if $j=2m+1$.}
\end{cases}
\end{align}

If $n=2m$ for $m\geq1$, then, for $0\leq i\leq3$ 
and $0\leq j\leq2m$, 
\begin{align}\label{image_2m}
	&\partial^{2m}(\mathfrak{p}_{i,j}^{2m}):=
\nonumber\\
	&\begin{cases}
\mathfrak{p}_{i,0}^{2m-1}x_{i+1}+x_{i}
\mathfrak{p}_{i+1,0}^{2m-1} 
	\quad\mbox{if $j=0$}\\[2mm]
S_{2m-j,i+j-1}\mathfrak{p}_{i,j-1}^{2m-1}
x_{i}^{4T+1}+\mathfrak{p}_{i,j}^{2m-1}x_{i+1}
	 \\
\quad\quad\quad\quad\quad+S_{j,i}x_{i}
\mathfrak{p}_{i+1,j}^{2m-1}+x_{i+1}^{4T+1}
\mathfrak{p}_{i+1,j-1}^{2m-1}\quad\mbox{if $1\leq 
j<m$}\\[2mm]
\sum_{k=0}^{T}\big[x_{i}^{4k}\big(S_{m,i+m-1}
\mathfrak{p}_{i,m-1}^{2m-1}x_{i}+S_{m,i}x_{i}
\mathfrak{p}_{i+1,m}^{2m-1}\big)x_{i}^{4T-4k}\\
\ +x_{i+1}^{4k}\big(\mathfrak{p}_{i,m}^{2m-1}x_{i+1}
+x_{i+1}\mathfrak{p}_{i+1,m-1}^{2m-1}\big)
x_{i+1}^{4T-4k}\big]\\
\ +\sum_{k=0}^{T-1}\big[x_{i}^{4k+2}\big(S_{m,i}
\mathfrak{p}_{i+2,m-1}^{2m-1}x_{i}+S_{m,i+m-1}x_{i}
\mathfrak{p}_{i+3,m}^{2m-1}\big)x_{i}^{4T-4k-2}
	 \\
\ +S_{m,i+3}^{-1}S_{m,i+m+2}x_{i+1}^{4k+2}
\big(\mathfrak{p}_{i+2,m}^{2m-1}x_{i+1}+x_{i+1}
\mathfrak{p}_{i+3,m-1}^{2m-1}\big)x_{i+1}^{4T-4k-2}
\big]\\
\hspace{7.5cm}\mbox{if $j=m$}\\[2mm]
S_{2m-j,i+j-1}\mathfrak{p}_{i,j-1}^{2m-1}x_{i}+
\mathfrak{p}_{i,j}^{2m-1}x_{i+1}^{4T+1}
	 \\
\quad\quad\quad+S_{j,i}x_{i}^{4T+1}
\mathfrak{p}_{i+1,j}^{2m-1}+x_{i+1}
\mathfrak{p}_{i+1,j-1}^{2m-1} 
	\quad\mbox{if $m+1\leq j<2m$}\\[2mm]
\mathfrak{p}_{i,2m-1}^{2m-1}x_{i}+x_{i+1}
\mathfrak{p}_{i+1,2m-1}^{2m-1}
	\quad\mbox{if $j=2m$.}
\end{cases}
\end{align}

\begin{remark}\label{resolution_remark}
%%%%%%%%%
%%%%%%%%%
\rm%%%%%%
%%%%%%%%%
%%%%%%%%%
\begin{enumerate}[(a)]
\item By direct computations, it is not hard to 
check that $\partial^{n}\partial^{n+1}=0$ for all 
$n\geq0$, and thus $(R^{\bullet},\partial)$ is a 
complex of $A$-$A$-bimodules.  
\item For $n\geq0$, the right $A$-homomorphism 
$h_{n}: A/\mathfrak{r}_{A}\otimes_{A}R^{n}\rightarrow 
P^{n}$determined by $\mathfrak{o}(g_{i,j}^{n})
\otimes_{A}\mathfrak{p}_{i,j}^{n}\mapsto\mathfrak{t}
(g_{i,j}^{n})$ $(0\leq i\leq3;\ 0\leq j\leq n)$ is 
an isomorphism, and the diagram 
$$\begin{CD}
A/\mathfrak{r}_{A}\otimes_{A}R^{n+1} & 
@>id_{A/\mathfrak{r}_{A}}\otimes_{A}
\partial^{n+1}>> & 
A/\mathfrak{r}_{A}\otimes_{A}R^{n} \\
@Vh_{n+1}V\simeq V & & @V\simeq Vh_{n}V \\
P^{n+1} & @>d^{n+1}>> & P^{n}
\end{CD}$$
is commutative. Therefore it follows that the 
complex $(A/\mathfrak{r}_{A}\otimes_{A}R^{\bullet},
id_{A/\mathfrak{r}_{A}}\otimes_{A}\partial)$ 
is a minimal projective resolution of the right 
$A$-module $A/\mathfrak{r}_{A}\otimes_{A}A$ 
$(\simeq A/\mathfrak{r}_{A})$. 
\end{enumerate}
\end{remark}
\noindent
Now, using Remark~\ref{resolution_remark}, we can 
provide a projective bimodule resolution of $A$: 

\begin{thm}\label{resolution_of_A}
The complex $(R^{\bullet},\partial)$ is a minimal 
projective bimodule resolution of $A=A_{T}(q_{0},
q_{1},q_{2},q_{3})$ for all $T\geq0$ and 
$q_{0},q_{1},q_{2},q_{3}\in K^{\times}$. 
\end{thm}
\noindent
The proof is similar to that found in some recent 
works, for example \cite{ES, F1, F2, SS, ST}, and 
so we omit it. 

\section{The Hochschild cohomology groups of $A$}
In this section, we compute the dimensions of the 
Hochschild cohomology groups of $A=A_{T}(q_{0},
q_{1},q_{2},q_{3})$ by using the projective bimodule 
resolution $(R^{\bullet},\partial)$ of 
Section~\ref{section_resolution}. Throughout this 
section, we keep the notation from the previous 
sections. 

Recall that, for $n\geq0$, the $n$th Hochschild 
cohomology group ${\rm HH}^n(A)$ of $A$ is defined  
to be the $K$-space ${\rm HH}^{n}(A):=
{\rm Ext}^{n}_{A^{e}}(A,A)$. We set $(-)^{*}:={\rm 
Hom}_{A^{\rm e}}(-,A)$, for simplicity. Then, for 
$n\geq0$, there is the following exact sequence of 
$K$-spaces:  
\begin{equation}\label{exact_sequence}
0\rightarrow({\rm Ker}\,\partial^{n-1})^{*}
\longrightarrow(R^{n})^{*}\overset{i_{n}^{*}}
\longrightarrow({\rm Ker}\,\partial^{n})^{*}
\longrightarrow {\rm HH}^{n+1}(A)\rightarrow0, 
\end{equation}
where ${\rm Ker}\,\partial^{-1}:=A$, and $i_{n}$ 
denotes the inclusion map. Hence, to give the 
dimensions of the Hochschild cohomology groups, it 
suffices to calculate the dimensions of $(R^{n})^{*}$ 
and $({\rm Ker}\,\partial^{n+1})^{*}$. 

\subsection{The dimension of ${\rm Hom}_{A^{\rm e}}
(R^{n},A)$}
We first find the dimension of $(R^{n})^{*}:=
{\rm Hom}_{A^{\rm e}}(R^{n},A)$ for $n\geq0$. 

\begin{defin}
Let $n\geq0$ be an integer. For $l=0,1$, $0\leq i
\leq3$ and $0\leq j\leq n$, and for $\begin{cases}
0\leq k\leq T&\mbox{if $n\not\equiv3$ 
$({\rm mod}\,4)$}\\ 0\leq k\leq T-1 & \mbox{if 
$T\geq1$ and $n\equiv3$ $({\rm mod}\,4)$,}\end{cases}$ 
we define the $A$-$A$-bimodule homomorphisms 
$\beta_{l,i,j}^{n,k}: R^{n}\rightarrow A$ by: for 
$0\leq r\leq3$ and $0\leq s\leq n$, 
$$\beta_{l,i,j}^{n,k}(\mathfrak{p}_{r,s}^{n}):=
\begin{cases}
e_{i}x_{l}^{4k+t} 
	& \mbox{if $r=i$, $s=j$, and $n\equiv t$ 
$({\rm mod}\,4)$ for $0\leq t\leq3$}\\
0 	& \mbox{otherwise.}
\end{cases}$$
\end{defin}
\noindent
Here we may consider the subscripts $l$ and $i$ of 
$\beta_{l,i,j}^{n,k}$ as modulo $2$ and $4$, 
respectively. Note that, for $u\geq0$ and $0\leq 
i\leq3$, $\beta_{0,i,j}^{4u,0}=\beta_{1,i,j}^{4u,0}$ 
($0\leq j\leq 4u$) and $\beta_{i+1,i,j}^{4u+2,T}=
-q_{i}\beta_{i,i,j}^{4u+2,T}$ ($0\leq j\leq4u+2$) 
hold. 

Now recall that, for $n\geq0$, the map 
$G_{n}:\bigoplus_{g\in\mathcal{G}^n}\mathfrak{o}(g)
A\mathfrak{t}(g)\rightarrow (R^{n})^{*}$ determined 
by $(G_{n}(\sum_{g\in \mathcal{G}^{n}}z_{g}))
(\mathfrak{p}_{i,j}^{n})=z_{g_{i,j}^{n}}$ for 
$z_{g}\in \mathfrak{o}(g)A\mathfrak{t}(g)$ $(g\in 
\mathcal{G}^n)$, $0\leq i\leq3$ and $0\leq j\leq n$, 
is an isomorphism of $K$-spaces. Also, for all 
$0\leq i\leq3$ and $0\leq j\leq n$, the space 
$\mathfrak{o}(g_{i,j}^{n})A\mathfrak{t}(g_{i,j}^{n})
=e_{i}Ae_{i+n}$ has a $K$-basis
$$\begin{cases}
\{e_{i},\ e_{i}x_{l}^{4k}\mid l=0,1;\ 1\leq k\leq 
T\} 
	& \mbox{if $n\equiv0$ $({\rm mod}\,4)$}\\
\{e_{i}x_{l}^{4k+1}\mid l=0,1;\ 0\leq k\leq T\} 
	&\mbox{if $n\equiv1$ $({\rm mod}\,4)$} \\
\{e_{i}x_{l}^{4k+2},\ e_{i}x_{0}^{4T+2}\mid 
l=0,1;\ 0\leq k\leq T-1\} 
	&\mbox{if $n\equiv2$ $({\rm mod}\,4)$} \\
\{e_{i}x_{l}^{4k+3}\mid l=0,1;\ 0\leq k\leq T-1\} 
	&\mbox{if $n\equiv3$ $({\rm mod}\,4)$ and 
$T\geq1$,} 
\end{cases}$$
and $\mathfrak{o}(g_{i,j}^{n})A\mathfrak{t}
(g_{i,j}^{n})=0$ if $n\equiv3$ $({\rm mod}\,4)$ and 
$T=0$. It is straightforward to see that the image 
of the $K$-basis above under $G_{n}$ is precisely 
the set: 
\begin{equation}\label{basis_P_n}
\begin{cases}
\{\beta_{0,i,j}^{n,0},\ \beta_{l,i,j}^{n,k}
\mid 0\leq i\leq3;\ 0\leq j\leq n;\ l=0,1;\ 1\leq k
\leq T\}\\
\hspace{4cm}\mbox{if $n\equiv0$ $({\rm mod}\,4)$}\\
\{\beta_{l,i,j}^{n,k}\mid 0\leq i\leq3;\ 0\leq j
\leq n;\ l=0,1;\ 0\leq k\leq T\}\\ 
\hspace{4cm}\mbox{if $n\equiv1$ $({\rm mod}\,4)$}\\
\{\beta_{l,i,j}^{n,k},\ \beta_{0,i,j}^{n,T}\mid 
0\leq i\leq3;\ 0\leq j\leq n;\ l=0,1;\ 0\leq k\leq 
T-1\}\\
\hspace{4cm}\mbox{if $n\equiv2$ $({\rm mod}\,4)$}\\
\{\beta_{l,i,j}^{n,k}\mid 0\leq i\leq3;\ 0\leq j\leq 
n;\ l=0,1;\ 0\leq k\leq T-1\} \\
\hspace{4cm}\mbox{if $n\equiv3$ $({\rm mod}\,4)$ 
and $T\geq1$,} 
\end{cases}
\end{equation}
and moreover $(R^{n})^{*}=0$ if $n\equiv3$ $({\rm mod}\,4)$ 
and $T=0$. Since the set (\ref{basis_P_n}) gives a 
$K$-basis of $(R^{n})^{*}$, we have the following: 
\begin{lem}\label{hom_dim}
For $m\geq0$ and $0\leq r\leq3$, 
$$\dim_{K}(R^{4m+r})^{*}=
\begin{cases}
4(2T+1)(4m+1) 	& \mbox{if $r=0$}\\
16(T+1)(2m+1) 	& \mbox{if $r=1$}\\
4(2T+1)(4m+3) 	& \mbox{if $r=2$}\\
32T(m+1) 	& \mbox{if $r=3$.}
\end{cases}$$
\end{lem}

\subsection{The dimension of ${\rm Hom}_{A^{\rm e}}
({\rm Ker}\,\partial^{n},A)$}
Until the end of this paper, we suppose that 
the product $S_{4,0}=q_{0}q_{1}q_{2}q_{3}$ $(\in 
K^{\times})$ is not a root of unity. Therefore
$S_{4m,t}^{k}=S_{4,0}^{mk}=(q_{0}q_{1}q_{2}
q_{3})^{km}\neq1_{K}$ for all $t\in \mathbb{Z}$ 
and integers $k\geq0$ and $m\geq0$. 

Now, using the equations (\ref{image_2m+1}) and 
(\ref{image_2m}), we can find a $K$-basis of the 
kernel of $i_{n}^{*}={\rm Hom}_{A^{\rm e}}
(i_{n},A)$ $(n\geq0)$ in (\ref{exact_sequence}). 

\begin{lem} \label{lem_basis}
If $T=0$, then 
\begin{enumerate}[\rm(i)]
\item \begin{enumerate}[\rm(a)]
\item The set $\{\sum_{j=0}^{3}\beta^{0,0}_{0,j,0}
\}$ gives a $K$-basis of ${\rm Ker}\,i_{0}^{*}$. 
\item For $m\geq1$, ${\rm Ker}\,i_{4m}^{*}=\{0\}$. 
\end{enumerate}

\item \begin{enumerate}[\rm(a)]

\item The set $\{\beta_{r,r,0}^{1,0}+\beta_{r+1,r,
1}^{1,0}-\beta_{r,r+1,1}^{1,0}-\beta_{r+1,r+1,
0}^{1,0}\mid r=0,1,2\}\cup\{\beta_{0,0,0}^{1,0}+
\beta_{1,0,1}^{1,0},\ \beta_{1,0,1}^{1,0}+
\beta_{0,1,1}^{1,0}+\beta_{1,2,1}^{1,0}+\beta_{0,3,
1}^{1,0}\}$ gives a $K$-basis of ${\rm Ker}\,i_{1}^{*}$. 

\item For $m\geq1$, the set $\{S_{u,r}\beta_{r,r,
u}^{4m+1,0}+\beta_{r+1,r,u+1}^{4m+1,0}-S_{4m-u,u+r
+1}\beta_{r,r+1,u+1}^{4m+1,0}-\beta_{r+1,r+1,u}^{4m
+1,0}\mid0\leq r\leq3;\ 0\leq u\leq4m\}$ gives a 
$K$-basis of ${\rm Ker}\,i_{4m+1}^{*}$. 
\end{enumerate}

\item For $m\geq0$, the set $\{\beta_{0,r,u}^{4m+2,
0}\mid0\leq r\leq3;\ 0\leq u\leq4m+2\}$ gives a 
$K$-basis of ${\rm Ker}\,i_{4m+2}^{*}$. 

\item For $m\geq0$, ${\rm Ker}\,i_{4m+3}^{*}=\{0\}$. 
\end{enumerate}
If $T\geq1$, then 
\begin{enumerate}[\rm(i)]

\item \begin{enumerate}[\rm(a)]

\item The set $\{\sum_{j=0}^{3}\beta_{0,j,0}^{0,
0}\}\cup\{\sum_{j=0}^{3}\beta_{l,j,0}^{0,k}\mid 
l=0,1;\ 1\leq k\leq T\}$ gives a $K$-basis of 
${\rm Ker}\,i_{0}^{*}$. 

\item For $m\geq1$, the set $\{S_{u,r}\beta_{r,r,
u}^{4m,k}+\beta_{r,r+1,u}^{4m,k}\mid0\leq r\leq
3;\ 0\leq u\leq2m-1;\ 1\leq k\leq T\}\cup\{\beta_{r,
r+1,u}^{4m,k}+S_{4m-u,r+u+1}\beta_{r,r+2,u}^{4m,k}
\mid0\leq r\leq3;\ 2m+1\leq u\leq4m;\ 1\leq k\leq 
T\}\cup\{S_{4m,0}\beta_{l,l,2m}^{4m,k}+S_{2m,l+2}
\beta_{l,l+1,2m}^{4m,k}+S_{2m,l+2}S_{2m,l+3}
\beta_{l,l+2,2m}^{4m,k}+S_{2m,l+3}\beta_{l,l+3,
2m}^{4m,k}\mid l=0,1;\ 1\leq k\leq T\}$ gives a 
$K$-basis of ${\rm Ker}\,i_{4m}^{*}$. 
\end{enumerate}

\item \begin{enumerate}[\rm(a)]
\item The set $\{\beta_{r,r,0}^{1,k},\ \beta_{r,
r+1,1}^{1,k}\mid0\leq r\leq3;\ 1\leq k\leq T\}\cup
\{\beta_{0,0,0}^{1,0}+\beta_{1,0,1}^{1,0}\}\cup 
\{\beta_{r,r,0}^{1,0}+\beta_{r+1,r,1}^{1,0}-
\beta_{r,r+1,1}^{1,0}-\beta_{r+1,r+1,0}^{1,0}
\mid r=0,1,2\}$ 
$$\qquad\qquad\qquad\cup\begin{cases}
\{\beta_{0,1,1}^{1,0}+\beta_{0,3,1}^{1,0},\ 
\beta_{1,0,1}^{1,0}+\beta_{1,2,1}^{1,0}\}
& \mbox{if ${\rm char}\,K\mid2T+1$}\\
\{\beta_{0,1,1}^{1,0}+\beta_{0,3,1}^{1,0}+\beta_{1,
0,1}^{1,0}+\beta_{1,2,1}^{1,0}\} & \mbox{if 
${\rm char}\,K\nmid2T+1$}
\end{cases}$$
gives a $K$-basis of ${\rm Ker}\,i_{1}^{*}$. 

\item For $m\geq1$, the set $\{S_{u,r}\beta_{r,r,
u}^{4m+1,0}+\beta_{r+1,r,u+1}^{4m+1,T}-S_{4m-u,u+r
+1}\beta_{r,r+1,u+1}^{4m+1,T}-\beta_{r+1,r+1,u}^{4m
+1,0}\mid0\leq r\leq3;\ 0\leq u\leq2m-1\}\cup\{
S_{u,r}\beta_{r,r,u}^{4m+1,T}+\beta_{r+1,r,u+1}^{4m
+1,0}-S_{4m-u,u+r+1}\beta_{r,r+1,u+1}^{4m+1,0}-
\beta_{r+1,r+1,u}^{4m+1,T}\mid0\leq r\leq3;\ 2m+1
\leq u\leq4m\}\cup\{\beta_{r,r,u}^{4m+1,k}\mid0
\leq r\leq3;\ 0\leq u\leq2m;\ 1\leq k\leq T\}\cup
\{\beta_{r+1,r,u}^{4m+1,k}\mid0\leq r\leq3;\ 
2m+1\leq u\leq4m+1;\ 1\leq k\leq T\}\cup
\{S_{2m,r}\beta_{r,r,2m}^{4m+1,0}+\beta_{r+1,r,
2m+1}^{4m+1,0}-S_{2m,r+3}\beta_{r,r+1,2m+1}^{4m+1,
0}-\beta_{r+1,r+1,2m}^{4m+1,0}\mid r=0,1\}$ 
$$\qquad\qquad\qquad\cup\begin{cases}
\{S_{2m,r+l}\beta_{r+l,r,2m+l}^{4m+1,0}+
S_{2m,r+l+3}\beta_{r+l,r+2,2m+l}^{4m+1,0}\\ 
\hspace{2cm}\mid r=0,1;\ l=0,1\}\quad\mbox{if 
${\rm char}\,K\mid2T+1$}\\
\{S_{2m,r}\beta_{r,r,2m}^{4m+1,0}+\beta_{r+1,r,
2m+1}^{4m+1,0}-S_{2m,r+3}\beta_{r,r+1,2m+1}^{4m+1,0}\\
\hspace{2cm}-\beta_{r+1,r+1,2m}^{4m+1,0}\mid r=2,3\}
\quad\mbox{if ${\rm char}\,K\nmid2T+1$}
\end{cases}$$
gives a $K$-basis of ${\rm Ker}\,i_{4m+1}^{*}$. 
\end{enumerate}

\item For $m\geq0$, the set $\{S_{u,r}\beta_{r,
r,u}^{4m+2,k}+\beta_{r,r+1,u}^{4m+2,k}\mid 
0\leq r\leq3;\ 0\leq u\leq2m;\ 0\leq k\leq T-1\} 
\cup\{\beta_{r,r+1,u}^{4m+2,k}+S_{4m-u+2,u+r+1}
\beta_{r,r+2,u}^{4m+2,k}\mid0\leq r\leq3;\ 2m+2
\leq u\leq4m+2;\ 0\leq k\leq T-1\}\cup\{S_{2m,l}
\beta_{l,l+3,2m+1}^{4m+2,k}+S_{2m,l}S_{2m+1,l+2}
\beta_{l,l+2,2m+1}^{4m+2,k}+S_{2m,l+3}\beta_{l,l+1,
2m+1}^{4m+2,k}+S_{4m+1,l}\beta_{l,l,2m+1}^{4m+2,k}
\mid l=0,1;\ 0\leq k\leq T-1\}\cup\{\beta_{0,r,
u}^{4m+2,T}\mid0\leq r\leq3;\ 0\leq u\leq4m+2\}$ 
gives a $K$-basis of ${\rm Ker}\,i_{4m+2}^{*}$. 

\item For $m\geq0$, the set $\{\beta_{r,r,u}^{4m+3,
k}\mid0\leq r\leq3;\ 0\leq u\leq2m+1;\ 0\leq k\leq 
T-1\}\cup\{\beta_{r,r+1,u}^{4m+3,k}\mid0\leq r\leq
3;\ 2m+2\leq u\leq4m+3;\ 0\leq k\leq T-1\}$
gives a $K$-basis of ${\rm Ker}\,i_{4m+3}^{*}$. 
\end{enumerate}
\end{lem}

\begin{proof}
We only verify (i)(b) for the case $T\geq1$. 
Suppose that $T\geq1$ and $m\geq1$, and we put 
$\mathcal{B}:=\{S_{u,r}\beta_{r,r,u}^{4m,k}+
\beta_{r,r+1,u}^{4m,k}\mid0\leq r\leq3;\ 0\leq 
u\leq2m-1;\ 1\leq k\leq T\}\cup\{\beta_{r,r+1,
u}^{4m,k}+S_{4m-u,r+u+1}\beta_{r,r+2,u}^{4m,k}\mid
0\leq r\leq3;\ 2m+1\leq u\leq4m;\ 1\leq k\leq T\}
\cup\{S_{4m,0}\beta_{l,l,2m}^{4m,k}+S_{2m,l+2}
\beta_{l,l+1,2m}^{4m,k}+S_{2m,l+2}S_{2m,l+3}
\beta_{l,l+2,2m}^{4m,k}+S_{2m,l+3}\beta_{l,l+3,
2m}^{4m,k}\mid l=0,1;\ 1\leq k\leq T\}$. Denote 
by $V$ the subspace in $(R^{4m})^{*}$ generated by 
$\mathcal{B}$. It is straightforward to see that the 
elements in $\mathcal{B}$ are linearly independent, 
and so it suffices to verify $V={\rm Ker}\,i_{4m}^{*}$. 

First, let $\psi\in\mathcal{B}$. Then it is not 
hard to check that $\psi$ sends all elements in 
(\ref{image_2m+1}) to zero, namely, $\psi
(\partial^{4m+1}(\mathfrak{p}_{i,j}^{4m+1}))=0$ for 
all $0\leq i\leq3$ and $0\leq j\leq 4m+1$. So 
we get $\psi({\rm Ker}\,\partial^{4m})=\psi
({\rm Im}\,\partial^{4m+1})=0$, which shows 
$\psi\in{\rm Ker}\,i_{4m}^{*}$, since $i_{4m}$ 
is the inclusion map. Hence it follows that 
$V\subseteq{\rm Ker}\,i_{4m}^{*}$. 

Conversely, let $\phi\in{\rm Ker}\,i_{4m}^{*}$. By 
(\ref{basis_P_n}), we can write $\phi=(\sum_{i=0}^{3}
\sum_{j=0}^{4m}v_{i,j}\beta_{0,i,j}^{4m,0})+
\{\sum_{i=0}^{3}\sum_{j=0}^{4m}\sum_{k=1}^{T}(w_{0,
i,j}^{k}\beta_{0,i,j}^{4m,k}+w_{1,i,j}^{k}\beta_{1,
i,j}^{4m,k})\}$ for $v_{i,j}$, $w_{l,i,j}^{k}\in K$ 
($0\leq i\leq3$, $0\leq j\leq4m$, $1\leq k\leq T$, 
$l=0,1$). We always consider the subscripts $i$ of 
$v_{i,j}$ and $w_{l,i,j}^{k}$ as modulo $4$, and 
$l$ of $w_{l,i,j}^{k}$ as modulo $2$. Since $\phi
\in{\rm Ker}\,i_{4m}^{*}$, we get $\phi({\rm Im}\,
\partial^{4m+1})=\phi({\rm Ker}\,\partial^{4m})=0$, 
so that $\phi(\partial^{4m+1}(\mathfrak{p}^{4m+1}_{r,
u}))=0$ for $0\leq r\leq3$ and $0\leq u\leq4m+1$. 
Then, by direct computations of the images of 
(\ref{image_2m+1}) under $\phi$, we get: for $0\leq 
r\leq3$, $(v_{r,0}-v_{r+1,0})e_{r}x_{r}+\sum_{k=1}^{T}
(w_{r,r,0}^{k}-w_{r,r+1,0}^{k})e_{r}x_{r}^{4k+1}=0$, 
$(S_{4m-u+1,r+u-1}v_{r,u-1}-v_{r+1,u-1})e_{r}
x_{r+1}^{4T+1}+(v_{r,u}-S_{u,r}v_{r+1,u})e_{r}x_{r}
+\sum_{k=1}^{T}(w_{r,r,u}^{k}-
S_{u,r}w_{r,r+1,u}^{k})e_{r}x_{r}^{4k+1}=0$ $(1\leq 
u\leq2m)$, $(S_{4m-u+1,r+u-1}v_{r,u-1}-v_{r+1,u-1})
e_{r}x_{r+1}+(v_{r,u}-S_{u,r}v_{r+1,u})e_{r}x_{r}^{4T
+1}+\sum_{k=1}^{T}(S_{4m-u+1,r+u-1}w_{r+1,r,u-1}^{k}-
w_{r+1,r+1,u-1}^{k})e_{r}x_{r+1}^{4k+1}$ $=0$ $(2m+1
\leq u\leq4m)$, $(v_{r,4m}-v_{r+1,4m})e_{r}x_{r+1}
+\sum_{k=1}^{T}(w_{r+1,r,4m}^{k}-w_{r+1,r+1,4m}^{k})$ 
$e_{r}x_{r+1}^{4k+1}=0$, which give the system of 
equations: for $0\leq r\leq3$ and $1\leq k\leq T$, 
$v_{r,u}-S_{u,r}v_{r+1,u}=0$ $(0\leq u\leq2m)$, 
$S_{4m-u,r+u}v_{r,u}-v_{r+1,u}=0$ $(0\leq u\leq2m-1)$, 
$w_{r,r,u}^{k}-S_{u,r}w_{r,r+1,u}^{k}=0$ $(0\leq 
u\leq2m)$, $S_{4m-u,r+u}v_{r,u}-v_{r+1,u}=0$ $(2m
\leq u\leq4m)$, $v_{r,u}-S_{u,r}v_{r+1,u}=0$ $(2m+1
\leq u\leq4m)$, $S_{4m-u,r+u}w_{r+1,r,u}^{k}-w_{r+1,
r+1,u}^{k}=0$ $(2m\leq u\leq4m)$. So it follows, for 
$0\leq r\leq3$ and $1\leq k\leq T$, that $v_{r,u}=0$ 
$(0\leq u\leq4m)$, $w_{r,r,u}^{k}=S_{u,r}w_{r,r+1,
u}^{k}$ $(0\leq u\leq2m-1)$, $S_{4m-u,r+u}w_{r+1,r,
u}^{k}=w_{r+1,r+1,u}^{k}$ $(2m+1\leq u\leq 4m)$, 
$S_{4m,0}w_{l,l+3,2m}^{k}=S_{2m,l}w_{l,l+2,2m}^{k}
=S_{2m,l+3}w_{l,l,2m}^{k}=S_{2m,l+3}S_{2m,l}w_{l,
l+1,2m}^{k}$ $(l=0,1)$. Then, using these equations, 
$\phi$ can be written as $\phi=\{\sum_{r=0}^{3}
\sum_{k=1}^{T}\sum_{u=0}^{2m-1}w_{r,r+1,u}(S_{u,r}
\beta_{r,r,u}^{4m,k}+\beta_{r,r+1,u}^{4m,k})\}+
\{\sum_{r=0}^{3}\sum_{k=1}^{T}\sum_{u=2m+1}^{4m}$
%%%%%%%%%%%%%%%%%%%%%%%%
%%%%%%%%%%%%%%%%%%%%%%%%
$w_{r,r+1,u}(\beta_{r,r+1,u}^{4m,k}+S_{4m-u,r+u+1}
\beta_{r,r+2,u}^{4m,k})\}+\{\sum_{k=1}^{T}
\sum_{l=0}^{1}S_{4m,0}^{-1}w_{l,l,2m}^{k}(S_{4m,0}
\beta_{l,l,2m}^{4m,k}+S_{2m,l+2}\beta_{l,l+1,2m}^{4
m,k}+S_{2m,l+2}S_{2m,l+3}\beta_{l,l+2,2m}^{4m,k}+
S_{2m,l+3}\beta_{l,l+3,2m}^{4m,k})\}$, which shows 
$\phi\in V$. Accordingly it follows that 
${\rm Ker}\,i_{4m}^{*}\subseteq V$. 

Therefore we get the required $K$-basis. 
Similarly, by using (\ref{image_2m+1}) and 
(\ref{image_2m}), we can verify the remaining 
statements. 
\end{proof}

\noindent
By the sequence (\ref{exact_sequence}), 
we get $\dim_{K}({\rm Ker}\,\partial^{n})^{*}
=\dim_{K}{\rm Ker}\,i_{n+1}^{*}$ for $n\geq-1$ 
(where ${\rm Ker}\,\partial^{-1}:=A$), so that 
Lemma~\ref{lem_basis} gives us the dimension of 
$({\rm Ker}\,\partial^{n})^{*}$: 

\begin{prop}\label{ker_dim}
For integers $m\geq-1$ and $0\leq r\leq3$ 
with $4m+r\geq-1$, 
\begin{align*}
&\dim_{K}({\rm Ker}\,\partial^{4m+r})^{*}\\
&=\begin{cases}
2T+1	& \mbox{if $m=-1$ and $r=3$}\\
8T+6 	& \mbox{if $m=r=0$ and ${\rm 
char}\,K\mid2T+1$}\\
8T+5	& \mbox{if $m=r=0$ and ${\rm 
char}\,K\nmid2T+1$}\\
2(8m+3)+8T(2m+1) 	& \mbox{if $m\geq1$, $r=0$ 
and ${\rm char}\,K\mid2T+1$}\\
4(4m+1)+8T(2m+1)	& \mbox{if $m\geq1$, $r=0$ 
and ${\rm char}\,K\nmid2T+1$}\\
4(4m+3)+2T(8m+5)& \mbox{if $m\geq0$ and $r=1$}\\
16T(m+1) 	& \mbox{if $m\geq0$ and $r=2$}\\
2T(8m+9) 	& \mbox{if $m\geq0$ and $r=3$.}
\end{cases}
\end{align*}
\end{prop}
\noindent
Now, by using the sequence (\ref{exact_sequence}) 
and Proposition~\ref{ker_dim}, we immediately get 
the dimensions of the Hochschild cohomology groups 
of $A=A_{T}(q_{0},q_{1},q_{2},q_{3})$: 
\begin{thm}\label{main_thm}
Let $T\geq0$ be an integer, and let $q_{i}$ $(0\leq 
i\leq3)$ be any element in $K^{\times}$. Suppose 
that the product $q_{0}q_{1}q_{2}q_{3}$ is not a 
root of unity. Then 
\begin{enumerate}[\rm(a)]

\item For $m\geq0$ and $0\leq r\leq3$, 
$$\dim_{K}{\rm HH}^{4m+r}(A)=
\begin{cases}
2T+1	& \mbox{if $m=r=0$}\\
2T+3 	& \mbox{if $m=0$, $r=1$ and ${\rm 
char}\,K\mid2T+1$}\\
2T+2	& \mbox{if $m=0$, $r=1$ and ${\rm 
char}\,K\nmid2T+1$}\\
2T+2 	& \mbox{if $m=0$, $r=2$ and ${\rm 
char}\,K\mid2T+1$}\\
2T+1	& \mbox{if $m=0$, $r=2$ and ${\rm 
char}\,K\nmid2T+1$}\\
2T 	& \mbox{if $m\geq0$ and $r=3$,}\\ 
	& \quad \mbox{or if $m\geq1$ and $r=0$}\\
2T+2	& \mbox{if $m\geq1$, $r=1$ and ${\rm 
char}\,K\mid2T+1$,}\\
	& \quad \mbox{or if $m\geq1$, $r=2$ and 
${\rm char}\,K\mid2T+1$}\\
2T 	& \mbox{if $m\geq1$, $r=1$ and ${\rm 
char}\,K\nmid2T+1$,}\\
	& \quad \mbox{or if $m\geq1$, $r=2$ and 
${\rm char}\,K\nmid2T+1$.}
\end{cases}$$
\item ${\rm HH}^{n}(A)=0$ for all $n\geq3$ if and 
only if $T=0$. 
\end{enumerate}
\end{thm}
\begin{remark} \rm 
If $T=0$, since the global dimension of $A_{0}
(q_{0},q_{1},q_{2},q_{3})$ is infinite for all 
$q_{i}\in K^{\times}$ $(0\leq i\leq3)$, 
Theorem~\ref{main_thm}~(b) gives us a negative 
answer to Happel's question stated in 
Section~\ref{introduction}. 
\end{remark}
\noindent
We end this paper by describing the Hochschild 
cohomology ring modulo nilpotence ${\rm HH}^{*}(A)
/\mathcal{N}_{A}$ of $A=A_{T}(q_{0},q_{1},q_{2},
q_{3})$ for $T=0$, where $\mathcal{N}_{A}$ is the 
ideal generated by all homogeneous nilpotent elements 
in ${\rm HH}^{*}(A)$, and thus ${\rm HH}^{*}(A)/
\mathcal{N}_{A}$ is a commutative graded algebra. 
As a consequence of Theorem~\ref{main_thm}, we get 
the following: 

\begin{cor}\label{Ho_vanish}
Let $T=0$ and $q_{i}\in K^{\times}$ for $0\leq 
i\leq3$. Suppose that $q_{0}q_{1}q_{2}q_{3}$ is not 
a root of unity. Then ${\rm HH}^{*}(A)$ is a 
$4$-dimensional local algebra, and ${\rm HH}^{*}(A)
/\mathcal{N}_{A}$ is isomorphic to $K$. 
\end{cor}


\begin{thebibliography}{****}
\bibitem[BE]{BE} P.~A.~Bergh and K.~Erdmann, {Homology 
and cohomology of quantum complete intersections}, 
Algebra Number Theory {\bf 2} (2008), 501--522. 

\bibitem[BGMS]{BGMS} R.-O.~Buchweitz, E.~L.~Green, 
D.~Madesen and \O.~Solberg, Finite Hochschild 
cohomology without finite global dimension, 
Math. Res. Lett. {\bf 12} (2005), 805-816. 

\bibitem[ES]{ES} K.~Erdmann and S.~Schroll, {On the 
Hochschild cohomology of tame Hecke algebras}, 
Arch. Math. (Basel) {\bf 94} (2010), 117--127. 


\bibitem[F1]{F1} T.~Furuya, A projective bimodule 
resolution and the Hochschild cohomology for a 
cluster-tilted algebra of type $\mathbb{D}_4$, 
SUT J. Math. {\bf 34} (2012), 145--169.

\bibitem[F2]{F2} \bysame, Hochschild cohomology 
for a class of some self-injective special 
biserial algebras of rank four, to appear in 
J.~Pure and Applied Algebra. 

\bibitem[GSZ]{GSZ} E.~L.~Green, \O.~Solberg and 
D.~Zacharia, {Minimal projective resolutions}, 
Trans. Amer. Math. Soc. {\bf 353} (2001), 
2915--2939. 

\bibitem[H]{H} D.~Happel, The Hochschild 
cohomology of finite-dimensional algebras, Springer 
Lecture Notes in Mathematics {\bf 1404} (1989), 
108--126. 

\bibitem[PS]{PS} A.~Parker and N.~Snashall, A 
family of Koszul self-injective algebras with finite 
Hochschild cohomology, J.~Pure and Applied Algebra 
{\bf 216} (2012), 1245-1252. 

\bibitem[SS]{SS} S.~Schroll and N.~Snashall, 
{Hochschild cohomology and support varieties for 
tame Hecke algebras}, Q.~J.~Math. {\bf 62} (2011), 
1017--1029. 

\bibitem[ST]{ST} N.~Snashall and R.~Taillefer, {The 
Hochschild cohomology ring of a class of special 
biserial algebras}, J.~Algebra Appl. {\bf 9} (2010), 
73--122. 

\bibitem[XZ]{XZ} Y.~Xu and C.~Zhang, More 
counterexamples to Happel's question and 
Snashall-Solberg's conjecture, arXiv:1109.3956. 
\end{thebibliography}
\end{document}